\documentclass{article}
\usepackage{url} 
\usepackage{amsmath,amssymb,amsthm}

\DeclareMathOperator{\gal}{Gal}
\DeclareMathOperator{\car}{char}

\newtheorem{thm}{Theorem}[section]
\newtheorem{lem}[thm]{Lemma}

\newtheorem{defn}[thm]{Definition}

\begin{document}

\title{ON THE PRE-IMAGE OF A POINT UNDER AN ISOGENY}

\author{JONATHAN REYNOLDS}


\maketitle


\begin{abstract} Given a rational point on a curve in a rational isogeny class, a natural question concerns the field of 
definition of its pre-images. The multiplication by $m$ endomorphism is a powerful and much-used tool. The pre-images for 
this map are found by factorizing a monic polynomial of degree $m^2$. For $m=2$, Everest and King gave examples where the 
existence of a quadratic factor coincided with the existence of a rational pre-image via a $2$-isogeny. Nelson Stephens 
asked if this always happens and the question is answered in the affirmative. It is also shown that the analogue for $m=3$
can only be false when there exists a rational point of order three and a small number of counterexamples are found. The 
results are proven over any field with characteristic not two or three.         
\end{abstract}



\section{Introduction}
Given an elliptic curve $E/ \mathbb{Q}$, the set of all curves $E'$ isogenous to $E$ over $\mathbb{Q}$ is finite (up to 
isomorphism) and is known as an isogeny class. V\'{e}lu's formulae \cite{1} and the Weierstrass parameterization of
the elliptic curve can be used to find an isogeny class. This is best illustrated in an algorithm developed by Cremona 
\cite{2}. He has used his algorithm to produce tables of isogeny classes \cite{3}. For each curve in the class, 
non-torsion generators of the Mordell-Weil group are also given. 
Methods which find isogeny classes over an arbitrary field have also been developed \cite{4}. 
Given a curve in an isogeny class and a point in the Mordell-Weil group, the focus here is on where the pre-images 
of the point are defined.

\begin{defn} 
Let $K$ be a field, $E/K$ an elliptic curve, $P \in E(K)$ and $\phi:E' \to E$ an isogeny. Suppose that $E'$, $\phi$ and
a point in $\phi^{-1}(P)$ are all defined over a finite extension $L/K$ with $[L:K]<\deg \phi$. Then $P$ is called 
\emph{maginfied}. If $L=K$ then $P$ is called \emph{K-magnified}.
\end{defn}

Suppose that $E/\mathbb{Q}$ is an elliptic curve. In \cite{5} Everest, 
Miller and Stephens showed that if a rational non-torsion point is $\mathbb{Q}$-magnified by $\phi:E' \to E$ then the 
corresponding elliptic divisibility sequence has only finitely many prime power terms. In \cite{6} it was shown 
that these finitely many terms correspond to $S$-integral points on $E'$, where $S$ is given explicitly. By the main 
result in \cite{7}, the 
$\mathbb{Q}$-magnified condition can be weakened to magnified when the isogeny is multiplication by two or three. 
Stephens asked if a non-torsion point is $\mathbb{Q}$-magnified whenever it is magnified by doubling a point. The 
following theorem resolves this question.      

\begin{thm} \label{1.2}
Let $E$ be an elliptic curve defined over a field $K$ with $\car{K} \ne 2$. If a non-torsion point $P \in E(K)$ is 
magnified by doubling a point then it is $K$-magnified.  
\end{thm}

It should be noted that factorizing a monic quartic polynomial determines whether a point is magnified by 
doubling a point (see Section \ref{2}). This polynomial depends only on the point and the coefficients of a Weierstrass 
equation for $E$. For multiplication by three the conclusion has to be weakened. 

\begin{thm} \label{1.3}
Let $E$ be an elliptic curve defined over a field $K$ with $\car{K} \ne 2,3$. If a non-torsion point $P \in E(K)$ is 
magnified by tripling a point then either it is $K$-magnified or $E$ has a non-trivial $K$-rational $3$-torsion point.
\end{thm}

Listed in Table~\ref{table1} are curves having a non-torsion rational point which is magnified by tripling a point but 
not $\mathbb{Q}$-magnified. They were found using PARI/GP \cite{11} and in Section~\ref{3} it is explained that such 
examples are rare. 

\section{Proofs of Theorems \ref{1.2} and \ref{1.3}} \label{2}
Let $K$ be a field with $\car{K} \ne 2$. Let $E/K$ be an elliptic curve with Weierstrass coordinate
functions $x$ and $y$. Let $P \in E(K)$ be a non-torsion point. The following observation plays an important role.

\begin{lem} \label{2.1}
Suppose that $E'/K$ is an elliptic curve and $\phi: E' \to E$ is an isogeny defined over $K$ with $\phi(R)=P$. Then 
$K(x(R), y(R))=K(x(R))$. 
\end{lem}

\begin{proof}
Put $L=K(x(R))$ and $L'=K(x(R), y(R))$. Then $[L':L] \le 2$. Suppose that $[L':L]=2$ and choose 
$\sigma \in \gal(L'/L)$ to be non-trivial. Then $T=\sigma(R)-R$ is in the kernel of $\phi$ since 
$\sigma(\phi(R))-\phi(R)= \mathcal{O}$. But $x(R+T)=x(R)$ so $R+T=\pm R$. Since $P$ is non-torsion it follows that 
$\sigma(R)=R$ and $L'=L$. 
\end{proof}

Let $\psi_m, \theta_m \in K[E]$ be the standard division polynomials (see p. 39 of \cite{9}). Define 
$\delta_m^P \in K[x]$ by $\delta_m^P=\theta_m-x(P)\psi_m^2$. Then $\delta_m^P$ is monic and has degree $m^2$. The 
zeros of $\delta_m^P$ determine the values of $x(R)$ for which $mR=P$.
   
\begin{proof}[Proof of Theorem \ref{1.2}] Assume that $P$ is magnified by doubling $R \in E(L)$. Using
Lemma \ref{2.1}, $L=K(x(R))$. We may choose $R$ so that $[L:K] \le 2$. Suppose that $[L:K]=2$ and choose 
$\sigma \in \gal(L/K)$ to be 
non-trivial. Then $T=\sigma(R)-R$ is a $2$-torsion point since $\sigma(2R)-2R=\mathcal{O}$. Also $T \in E(K)$ since 
$\sigma(T)=-T$. Using this torsion point, we can construct an elliptic curve $E'/K$ and a $2$-isogeny $\phi: E \to E'$ 
with $\ker \phi=\{ \mathcal{O},T \}$ (see p. 95 of \cite{10}). Moreover, both $\phi$ and its dual $\hat{\phi}: E' \to E$ 
are defined over $K$. Put $\phi(R)=Q$. It follows that $\sigma(Q)=\phi(\sigma(R))=\phi(R+T)=\phi(R)$. Hence $Q \in E'(K)$ 
and $\hat{\phi}(Q)=P$.   
\end{proof}

\begin{proof}[Proof of Theorem \ref{1.3}]
We may assume that $\delta_3^P$ has an irreducible factor $f$ of degree $3$. Let $r_1,r_2,r_3$ be the zeros of $f$ and 
let $R_1,R_2,R_3$ be the corresponding solutions to $3R=P$. Then the splitting field $L$ of $f$ is $K(r_1, \sqrt{D})$, 
where $D$ is the discriminant of $f$. Moreover, $\gal(L/K)$ is isomorphic to $A_3$ or $S_3$. Let $\alpha=(1,2,3)$ be the
generator of $\gal(L/K(\sqrt{D}))$ and put $T=\alpha(R_1)-R_1$. Suppose that $E$ does not have a non-trivial $K$-rational 
$3$-torsion point. Any permutation preserves $R_1+R_2+R_3-P$ so it must equal $\mathcal{O}$. Thus 
$\alpha(T)=R_3-R_2=R_2-R_1=T$ and $D$ is not a square in $K$. Choose $\sigma \in \gal(K(\sqrt{D})/K)$ to be non-trivial. 
Since $\sigma(T)+T$ is a $K$-rational $3$-torsion point, $x(T) \in K$. Using V\'{e}lu's formulae \cite{1}, we can 
construct an elliptic curve $E'/K$ and a $3$-isogeny $\phi: E \to E'$ with kernel $\{ \mathcal{O},T,-T \}$. Moreover, 
both $\phi$ and its dual $\hat{\phi}: E' \to E$ are defined over $K$. Now $\phi(R_2)=\phi(R_1+T)$ and 
$\phi(R_3)=\phi(R_1-T)$. Hence $\phi(R_1)$ is fixed by $\gal(L/K)$.           
\end{proof}

\section{Computations} \label{3}
Examples where $\delta_3^P$ factorizes (over $K$) but $P$ is not $K$-magnified are relatively hard to find.
Let $K=\mathbb{Q}$. The first $12$ of Cremona's ``generators'' tables from \cite{3} were considered. 
There are $22,962$ pairs $(E,P)$ such that $\delta_3^P$ factorizes (over $\mathbb{Q}$). 
In all but $14$ of these pairs $P$ is $\mathbb{Q}$-magnified by a $3$-isogeny, which can be constructed using V\'{e}lu's 
formulae \cite{1} and a rational zero of $\psi_3$.

\begin{table}[h]
\caption{Curves having magnified points which are not $\mathbb{Q}$-magnified}
\hbox{
\vtop{\hsize=1.5in
\begin{center}
\begin{tabular}{|c|c|c|}
\hline
$N$ & $C$ & \# \\
\hline
17739 & g & 1 \\
19926 & l & 2 \\
26730 & y & 2 \\
39710 & z & 1 \\
45662 & h & 1 \\
\hline
\end{tabular}
\end{center}
}
\vtop{\hsize=1.7in
\begin{center}
\begin{tabular}{|c|c|c|}
\hline
$N$ & $C$ & \# \\
\hline
47526 & f & 1 \\
49818 & j & 1 \\
57222 & bw & 2 \\
62814 & r & 1 \\
64395 & f & 1 \\
\hline
\end{tabular}
\end{center}
}
\vtop{\hsize=1.5in
\begin{center}
\begin{tabular}{|c|c|c|}
\hline
$N$ & $C$ & \# \\
\hline
70470 & m & 2 \\
92055 & u & 1 \\
113866 & d & 1 \\
119646 & dd & 2 \\ [0.6ex]
\hline
\end{tabular}
\end{center}
}
}
\label{table1}
\end{table}  

The $14$ counterexamples found are listed in Table \ref{table1}. They are given in the same format as Cremona uses, where 
$N$ is the conductor, $C$ is the isogeny class and \# is the number of the curve in the class. Each curve listed has rank 
$1$ and $P$ is taken to be the generator which Cremona gives. Moreover, they each lie in an isogeny class of size two. Let
$Q$ be a generator for the other curve. Denote by $\hat{h}$ the canonical height (see \cite{10}). If $P$ is 
$\mathbb{Q}$-magnified by a $d$-isogeny then it is the image of $mQ+T$, for some non-zero integer $m$ and torsion point 
$T$. This implies $dm^2\hat{h}(Q)=\hat{h}(P)$. But in fact, $\hat{h}(Q)=3\hat{h}(P)$. Thus $P$ is not 
$\mathbb{Q}$-magnified. 

It is worth noting that for all of the pairs $(E,P)$ considered and for all odd primes $l \le 7$, if $\delta_l^P$ 
factorizes (over $\mathbb{Q}$) then $\psi_l$ has a factor of degree at most $(l-1)/2$.

\end{document}